\documentclass[12pt,english,reqno]{amsart}
\usepackage{geometry}
\usepackage{amsmath,amssymb,amsthm,bbm,mathtools,comment}
\usepackage{amsfonts}
\usepackage[shortlabels]{enumitem}
\usepackage[pdftex,colorlinks,backref=page,citecolor=blue]{hyperref}
\usepackage[mathscr]{euscript}
\usepackage[usenames,dvipsnames]{color}
\usepackage{tikz,babel,adjustbox}
\usepackage[numbers]{natbib}
\usepackage{graphicx}
\usepackage{caption}
\usepackage{subcaption}
\usepackage{verbatim}
\usepackage{array}
\usepackage[frame,cmtip,arrow,matrix,line,graph,curve]{xy}
\usepackage{graphpap, pstricks}
\usepackage{etoolbox}
\usepackage{pifont}
\usepackage[final]{microtype}
\usepackage{cmtiup}

\hypersetup{pdfpagemode=UseNone,pdfstartview={XYZ null null 1.00}}
\usetikzlibrary{shapes.misc,calc,intersections,patterns,decorations.pathreplacing}
\usetikzlibrary{arrows,arrows.meta,shapes,positioning,decorations.markings}
\tikzstyle arrowstyle=[scale=1]

\allowdisplaybreaks

\makeatletter
\def\@settitle{\begin{center}%
		\bfseries\Large
		\@title
	\end{center}%
}
\patchcmd{\@setauthors}{\MakeUppercase}{\normalsize}{}{}
\makeatother

\theoremstyle{plain}
\newtheorem{thm}{Theorem}[section]      
\newtheorem{lemma}[thm]{Lemma}
\newtheorem{claim}[thm]{Claim}
\newtheorem{prop}[thm]{Proposition}
\newtheorem{cor}[thm]{Corollary}
\newtheorem{conj}[thm]{Conjecture}

\newtheorem{definition}[thm]{Definition}
\newtheorem{quest}[thm]{Question}
\theoremstyle{remark}

\newcommand{\beq}[1]{\begin{equation}\label{#1}}
	\newcommand{\enq}[0]{\end{equation}}

\newcommand{\CC}{{\mathbb C}}
\newcommand{\RR}{{\mathbb R}}

\renewcommand{\Pr}{{\mathbb P}}
\newcommand{\EE}{{\mathbb E}}
\newcommand{\eps}{\epsilon}

\newcommand{\sub}{\subseteq}

\begin{document}
	
	\title{On the reverse Littlewood--Offord problem of Erd\H{o}s}
	
	\author{Xiaoyu He}
	\address{School of Mathematics, Georgia Institute of Technology, Atlanta, GA 30332, USA}
	\email{xhe399@gatech.edu}
	
	\author{Tomas Ju\v skevi\v cius}
	\address{Institute of Computer Science, Vilnius University, Didlaukio 47, LT-08303 Vilnius, Lithuania}
	\email{tomas.juskevicius@gmail.com}
	
	\author{Bhargav Narayanan}
	\address{Department of Mathematics, Rutgers University, Piscataway, NJ 08854, USA}
	\email{narayanan@math.rutgers.edu}
	
	\author{Sam Spiro}
	\address{Department of Mathematics, Rutgers University, Piscataway, NJ 08854, USA}
	\email{sas703@scarletmail.rutgers.edu}
	
	\begin{abstract}
		Let $\epsilon_{1},\ldots,\epsilon_{n}$ be a sequence of independent Rademacher random variables.  Answering a question of Erd\H{o}s from 1945, Beck proved in 1983 --- using techniques from harmonic analysis --- that there is a constant $c>0$ such that for any unit vectors $v_1,\ldots,v_n\in \mathbb{R}^2$, we have
		$$\Pr\left[||\epsilon_1 v_1+\ldots+\epsilon_n v_n||_2 \leq \sqrt{2}\right]\geq \frac{c}{n}.$$
		We give a new, elementary proof of this result using a simple pairing argument that might be of independent interest.
	\end{abstract}
	
	\maketitle
	
	\section{Introduction}
	Broadly speaking, Littlewood--Offord theory asks for estimates on the number of subset sums of a given sequence $V$ of vectors $v_1,\ldots,v_n$ that lie within a target set $S$. This is equivalent to studying the probability that the random signed sum
	\[\sigma_V = \eps_1 v_1+\eps_2v_2+\cdots +\eps_n v_n\]
	lands within a given set, where the $\eps_i$ are independent Rademacher random variables (i.e., independent random variables with $\Pr\left[\epsilon_i = - 1\right]=\Pr\left[\epsilon_i = +1\right]={1}/{2}$).  Littlewood and Offord~\cite{littlewood1939number} originally considered the special case of this problem when each $v_i$ is a complex number of norm at least one and showed that the probability that $\sigma_V$ lies within any open ball of radius one is at most $O(n^{-1/2}\log n)$.  This result was sharpened in seminal work of Erd\H{o}s~\cite{ELO} who used Sperner's theorem to prove that this probability is at most ${n\choose \lfloor n/2\rfloor} 2^{-n}$, which is sharp when $v_i=1$ for all $1 \le i \le n$.
	
	A large amount of work has since been done on extending these classical results, both for `forward' Littlewood--Offord problems like those described above, as well as `inverse' Littlewood-Offord problems in the vein of Tao and Vu~\cite{tao2009inverse} (where one seeks a structural characterization of the vectors $V$ given that $\sigma_V$ is likely to land in $S$).  In addition to being interesting questions in their own right, both types of problems have garnered a great deal of attention due to their many applications in random matrix theory; see, for example~\cite{kahn1995probability,rudelson2008littlewood,tao2008random,tao2009inverse}. Our focus here will be on `reverse' Littlewood--Offord problems that ask for \emph{lower bounds} on the probability that $\sigma_V$ lies within a given target $S$.  Notable examples of such problems include Koml\'os's Conjecture~\cite{spencer1994ten} and Tomaszewski’s Conjecture~\cite{guy1986any}, the latter of which was recently resolved in breakthrough work of Keller and Klein~\cite{KK}.
	
	Erd\H{o}s~\cite{ELO} posed two natural conjectures in his original 1945 paper on the topic.  The first of these asked for an extension of his upper bound of ${n\choose \lfloor n/2\rfloor} 2^{-n}$ for complex numbers of norm at least 1 to vectors of norm at least 1 in arbitrary Hilbert spaces; this was eventually resolved in full by Kleitman~\cite{kleitman1970lemma} (who further extended this bound to arbitrary normed spaces).  Erd\H{o}s' second conjecture, which asks for generally applicable \emph{lower bounds}, is a reverse Littlewood--Offord problem: is it always true that a random signed sum of complex numbers of norm one is fairly likely to fall inside a closed unit ball centred at the origin?  More precisely, for $x_i \in \CC$ and $\eps_i\in \{-1,+1\}$, he raised the following problem:
	\begin{figure}[h!]
		\centering
		\includegraphics[width=0.8\textwidth]{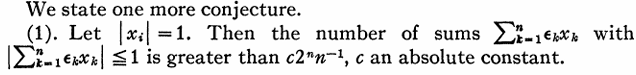}
	\end{figure}
	
	Equivalently, this conjecture states that if $V=(v_1,\ldots,v_n)$ is a sequence of unit vectors in $\mathbb{R}^2$, then $\Pr[\|\sigma_V\|_2 \le 1] = \Omega(n^{-1})$.  It was observed by Carnielli and Carolino~\cite{CC} that Erd\H{o}s' conjecture requires a minor adjustment (and is false as stated): assume that $n$ is even and consider $v_1=(1,0)$ and $v_i=(0,1)$ for all $i>1$. Each coordinate of $\sigma_V$ has absolute value at least $1$, and thus $\sigma_V$ has length at least $\sqrt{2}$.
	
	In view of this counterexample, Carnielli and Carolino adjusted Erd\H{o}s' conjecture by replacing $1$ by $\sqrt{2}$.  The main result of this paper is a new proof of this (adjusted) conjecture.
	
	\begin{thm}\label{thm:main}
		There exists an absolute constant $c>0$ such that for any unit vectors $v_1,\ldots,v_n\in \mathbb{R}^2$ and independent Rademacher random variables $\eps_1,\dots,\eps_n$, we have
		\[\Pr\left[\|\epsilon_1 v_1+\dots+\epsilon_n v_n\|_2 \leq \sqrt{2}\right]\geq \frac{c}{n}.\]
	\end{thm}
	
	Some time after we found a proof of Theorem~\ref{thm:main}, we discovered that this particular conjecture was solved in the following strong form by Beck~\cite{beck1983geometric} in a somewhat obscure (since it does not reference~\cite{ELO}) paper from 1983.
	
	\begin{thm}\label{thm:Beck}
		For all $d\ge 2$, there exists some $c_d>0$ such that for any unit vectors $v_1,\ldots,v_n\in \mathbb{R}^d$ and independent Rademacher random variables $\eps_1,\dots,\eps_n$, we have
		$$\Pr\left[\|\epsilon_1 v_1+\ldots+\epsilon_n v_n\|_2 \leq \sqrt{d}\right]\geq \frac{c_d}{n^{d/2}}.$$
	\end{thm}
	
	While Theorem~\ref{thm:main} is just a special case of Beck's theorem, we believe that our new proof has some independent value since, in particular, it uses elementary, geometric methods which (in our opinion) are much simpler than the deep, harmonic-analysis arguments used by Beck.
	
	The rest of this paper is organised as follows. The high-level ideas that form the basis of our arguments are encapsulated in a few key lemmas in Section~\ref{sec:prelim}, and the proof of Theorem~\ref{thm:main} then follows in Section~\ref{sec:2d}. We conclude with a discussion of open problems in Section~\ref{sec:conc}.
	
	\section{The Pairing Argument}\label{sec:prelim}
	For ease of notation, we henceforth write $\|\cdot \|$ for the Euclidean norm unless stated otherwise. Also, given a sequence $V=(v_1,\ldots, v_n)$ of vectors, we exclusively use $\sigma_V$ to denote the random signed sum $\epsilon_1 v_1+\dots +\epsilon_n v_n$.  Although we will ultimately only work in $\RR^2$, we state our lemmas here in terms of $\RR^d$ in general since this introduces no extra complications in our argument.
	
	To prove concentration of the random signed sum $\sigma_V \in \RR^d$ around the origin, it is natural to try and apply the second moment method (and Chebyshev's inequality in particular).  As was formally worked out in~\cite{CC}, this approach easily shows that $\sigma_V$ has a constant probability of landing within a ball of radius roughly $\sqrt{n}$, after which a pigeonholing argument implies that there is \textit{some} ball of constant radius in which $\sigma_V$ lands with probability $\Omega(n^{-d/2})$. However, there is no guarantee with this approach that this constant-radius ball is centered at the origin, and all of what we do in the sequel is aimed at circumventing this obstacle.
	
	We get around the obstacle described above by relating concentration estimates for $\sigma_V$ to concentration estimates for the difference of \emph{two independent copies} of $\sigma_V$. This reduction ultimately yields the following pairing lemma, the proof of which will be the main goal of this section, establishing sufficiently strong concentration for the random signed sum $\sigma_V$ provided one can find a reordering of our vectors $V = (v_1, \dots, v_n)$ such that the norms of the consecutive differences $v_{2i-1}-v_{2i}$ are small.
	\begin{prop}\label{prop:pairing}
		Let $V=(v_1,\ldots,v_n)$ be a sequence of unit vectors in $\RR^d$ with $n$ even and $r,\alpha>0$ reals such that $r^2\ge \alpha+\sum_{i=1}^{n/2} \|v_{2i-1}-v_{2i}\|^2$.  Then
		\[\Pr[\|\sigma_V\|\le r]=\Omega_{d,\alpha,r}\left(n^{-d/2}\right).\]
	\end{prop}
	
	Before we can prove Proposition~\ref{prop:pairing}, we require a few definitions.  Given a sequence of vectors $V=(v_1,\ldots,v_n)$, we define its sequence of difference vectors $\delta(V)$ by \[\delta(V)=(v_1-v_2,v_3-v_4,\ldots,v_{2\lfloor n/2\rfloor-1}- v_{2\lfloor n/2\rfloor}),\]
	i.e., $\delta(V)$  consists of all the differences $v_{2i-1}-v_{2i}$ for $1\le i\le \lfloor n/2\rfloor$.
	For a real number $a>0$ and a sequence of vectors $V$, we define
	\[p_a(V)=\Pr[\|\sigma_V\|\le a]\]
	and we define $q_a(V)$ to be the probability that two independent samples $X,X'$ of $\sigma_V$ satisfy $\|X - X'\| \le a$. The key ingredient for proving Proposition~\ref{prop:pairing} is the following lemma.
	
	\begin{lemma}\label{lem:iterate}
		Let $V=(v_1,\ldots, v_n)$ be a sequence of vectors in $\RR^d$, let $a,b>0$ be reals, and set
		\[\overline{V} = \left(\frac{1}{2}(v_{2i-1} + v_{2i})\right)_{i=1}^{\lfloor n/2\rfloor}.\]  If $n$ is even, then
		$$
		p_{a+b}(V) \ge q_a(\overline{V}) \cdot \min_{D\subseteq \delta(V)} p_{b}(D),
		$$
		where $D$ ranges over all subsequences of $\delta(V)$.  If $n$ is odd and if every vector of $V$ has norm at most $K$, then
		$$
		p_{a+b+K}(V) \ge q_a(\overline{V}) \cdot \min_{D\subseteq \delta(V)} p_{b}(D).
		$$
	\end{lemma}
	\begin{proof}
		The odd case follows immediately from the even case since the probability that $\|\sigma_V\|\le c+ \max_{i}\|v_i\|$ is always at most the probability that $\|\sigma_{V-\{v_n\}}\|\le c$, so we assume in what follows that $n$ is even.
		
		We sample a random signed sum $\sum \eps_i v_i$ as follows. First, sample i.i.d.\  uniform random signs $\eps_i^{(1)}$, $\eps_i^{(2)}$ for all $i\le n/2$. Define $I$ to be the set of $i\le n/2$ for which $\eps_i^{(1)} = \eps_i^{(2)}$, and sample i.i.d.\  uniform random signs $\eps_i^{(3)}$ for each $i\in I$. Define
		\begin{align*}
			\sigma^{(1)} & = \sum_{i=1}^{n/2} \eps_i^{(1)} \cdot \frac{1}{2}(v_{2i-1} + v_{2i}) \\
			\sigma^{(2)} & = \sum_{i=1}^{n/2} \eps_i^{(2)} \cdot \frac{1}{2}(v_{2i-1} + v_{2i}) \\
			\sigma^{(3)} & = \sum_{i \in I} \eps_i^{(3)} \cdot (v_{2i-1} - v_{2i}).
		\end{align*}
		Observe that $\sigma^{(1)} - \sigma^{(2)}+\sigma^{(3)}$ has the same distribution as $\sigma_V$; this is easy to verify by checking  that each of the four possible signed sums of $v_{2i-1},v_{2i}$ are equally likely for every $i$. Thus,
		\begin{align*}
			p_{a+b}(V) & = \Pr[\|\sigma^{(1)} - \sigma^{(2)}+\sigma^{(3)}\| \le a+b]                                                                                  \\
			& =\EE_{D\subseteq \delta(V)} \left[ \Pr[\|\sigma^{(1)} - \sigma^{(2)}+\sigma^{(3)}\| \le a+b \,\mid\, I=D]\right]                             \\
			& \ge \EE_{D\subseteq \delta(V)}\left[\Pr[\|\sigma^{(1)}-\sigma^{(2)}\|\le a\,\mid\,I=D]\cdot \Pr[\|\sigma^{(3)}\| \le b \,\mid\, I=D] \right] \\
			& \ge \Pr[\|\sigma^{(1)}-\sigma^{(2)}\| \le a] \cdot \min_{D\subseteq \delta(V)}\Pr[\|\sigma^{(3)}\| \le b \,\mid\, I=D]                       \\
			& = q_a(\overline{V}) \cdot \min_{D\subseteq \delta(V)} p_{b}(D),
		\end{align*}
		as desired.
	\end{proof}
	
	To get good bounds from Lemma~\ref{lem:iterate}, it then suffices to establish good lower bounds on $q_a$, and to show that the sequence $V$ can be reordered in such a way that the vectors of $\delta(V)$ are small. The former is a straightforward consequence of Chebyshev's inequality.
	
	\begin{lemma}\label{lem:Chebyshev}
		If $W = (w_1,\ldots, w_n)$ is a sequence of $n$ vectors in $\RR^d$ of norm at most $K$, then $q_a(W) = \Omega_{d,a,K}( n^{-d/2})$ for any $a > 0$.
	\end{lemma}
	\begin{proof}
		For each coordinate $1 \le i \le d$, we have $\mathrm{Var}[(\sigma_W)_i] \le K^2n$, so by Chebyshev's inequality,
		we have
		\[
		\Pr[|(\sigma_W)_i| < 2dK\sqrt{n}] > 1-  \frac{1}{2d}.
		\]
		By taking a union bound over all $d$ dimensions, we obtain
		$$
		\Pr[\|\sigma_W\|_\infty < 2dK \sqrt{n}] > \frac{1}{2}.
		$$
		Thus there is at least a $1/2$ chance that $\sigma_V$ falls into a cube centered at the origin with side length $4dK\sqrt{n}$. Such a cube can be covered by $O_d((K\sqrt{n}/a)^d)$ subcubes of diameter $a$, so we see that if $\sigma_W$ and $\sigma_{W'}$ are independently sampled from the same distribution, the chance that they both fall into the same subcube is at least $\Omega_d( (a/K\sqrt{n})^d)$, as desired.
	\end{proof}
	Since the vectors of $\overline{W}=(\frac{1}{2}(w_1+w_2),\ldots)$ have norms no larger than the maximum norm of $W$, we immediately get the following corollary.
	\begin{cor}\label{cor:Chebyshev}
		If $W = (w_1,\ldots, w_n)$ is a sequence of $n$ vectors in $\RR^d$ of norm at most $K$, then $q_a(\overline{W}) = \Omega_{d,a,K}(n^{-d/2})$ for any $a > 0$.  \qed
	\end{cor}
	
	Similarly, once we know that the vectors of a difference sequence $\delta(V)$ are small on average, we can apply the following with $\delta(V)=W$.
	\begin{lemma}\label{lem:smallV}
		If $W=(w_1,\ldots,w_n)$ is a sequence of vectors in $\RR^d$ and $b,c>0$ are real numbers such that $\sum_{i=1}^n \|w_i\|^2\le b^2-c$, then for any subsequence $D$ of $W$ we have $p_b(D)\ge c/b^2$.
	\end{lemma}
	\begin{proof}
		Observe that
		\begin{align*}
			\EE[\|\sigma_D\|^2] & = \EE\left[\left\langle \sum_{u\in D} \eps_u u,\sum_{u\in D} \eps_u u \right\rangle \right]=\EE\left[\sum_{u,u'\in D}\eps_u \eps_{u'} \langle u,u'\rangle\right] \\
			& =\sum_{u\in D} \|u\|^2\le \sum_{i=1}^{n}\|w_{i}\|^2\le b^2-c,
		\end{align*}
		so by Markov's inequality, we find
		$$\Pr[\|\sigma_D\| \le b] = \Pr[\|\sigma_D\|^2 \le b^2] \ge 1-\frac{b^2-c}{b^2}=\frac{c}{b^2},$$
		giving the result.
	\end{proof}
	
	We now have all we need to prove Proposition~\ref{prop:pairing}.
	\begin{proof}[Proof of Proposition~\ref{prop:pairing}]
		Applying Lemma~\ref{lem:iterate} with $a=\alpha/3r$ and $b=r-a$ (which is positive since $r^2\ge \alpha>\alpha/3=ar$) gives
		\[p_r(V)\ge q_{a}(\overline{V}) \cdot \min_{D\sub \delta(V)} p_{r-a}(D).\]
		By Corollary~\ref{cor:Chebyshev} and the fact that $V$ consists of unit vectors, we get \[q_a(\overline{V})=\Omega_{d,a}(n^{-d/2})=\Omega_{d,\alpha,r}(n^{-d/2}).\]  For the second term, we apply Lemma~\ref{lem:smallV} with $b=r-a$ and $c=ra$ (which is valid since $(r-a)^2-ra\ge r^2-3ra=r^2-\alpha$) to get
		\[p_{r-a}(D)\ge ra/(r-a)^2\ge \alpha/3r^2\] for all $D$.  Putting these two estimates together gives the result.
	\end{proof}
	
	\section{Proof of the main result}\label{sec:2d}
	As noted earlier, the main obstacle in applying Lemma~\ref{lem:iterate} to show concentration of $\sigma_V$ is to find an effective reordering of $V$ so that the sequence $\delta(V)=(v_1-v_2,\ldots)$ has small norms.  In the two-dimensional case, there is a natural way to do this, namely by ordering the vectors by argument as they are arranged  around the unit circle.  For this, given a unit vector $v$ in $\mathbb{R}^2$, we write $\arg(v)$ to denote the unique $\theta$ with $0\le \theta<2\pi$ such that $v=(\cos(\theta),\sin(\theta))$.
	\begin{lemma}\label{lem:circleOrder}
		Let $V=(v_1,\ldots,v_{n+1})$ be unit vectors in $\RR^2$ with $v_1=(1,0)$ and $v_{n+1}=(-1,0)$ such that $0=\arg(v_1)\le \arg(v_2)\le \cdots \le \arg(v_{n+1})=\pi$.  Then $\sum_{i=1}^n \|v_i-v_{i+1}\|^2\le 4$.
	\end{lemma}
	In fact, we will need the following generalization which recovers Lemma~\ref{lem:circleOrder} by taking $V'=((1,0),(-1,0))$.
	\begin{lemma}\label{lem:generalizedCircleOrder}
		Let $V'=(v_1',\ldots,v_{m}')$ be unit vectors in $\RR^2$ with $0=\arg(v_1')\le \arg(v_2')\le \cdots \le \arg(v_m')\le \pi$.  If $V=(v_1,\ldots,v_{n+1})$ is a sequence of unit vectors in $\RR^2$ which contains $V'$ as a subsequence and which satisfies $0=\arg(v_1)\le \arg(v_2)\le \cdots \le \arg(v_{n+1})=\arg(v'_m)$, then \[\sum_{i=1}^n \|v_i-v_{i+1}\|^2\le \sum_{i=1}^{m-1} \|v_i'-v_{i+1}'\|^2.\]
	\end{lemma}
	\begin{proof}
		Observe that the points $v_i$ all lie in the semicircle $0\le \arg(v_i) \le \pi$, so in any triangle $v_iv_{i+1}v_{i+2}$ the angle at $v_{i+1}$ is either obtuse or right. Thus, \[\|v_i-v_{i+1}\|^2 + \|v_{i+1}-v_{i+2}\|^2 \le \|v_i-v_{i+2}\|^2.\] The result then follows by iteratively removing the terms that are in $V$ but not $V'$ (since the inequality above implies that this procedure can never decrease the sum $\sum \|v_i-v_{i+1}\|^2$).
	\end{proof}

	With Lemma~\ref{lem:circleOrder}, we can already give a short proof of a slightly weakened version of Theorem~\ref{thm:main} when $n$ is even.  We emphasize that this result is not needed for our main argument, but it does serve as a warm-up to our more general approach that necessitates some more careful geometric estimates.
	
	\begin{prop}\label{prop:weak2}
		For any $r>\sqrt{2}$, there exists an absolute constant $c=c(r) > 0$ such that if $V=(v_1,\ldots,v_n)$ is a sequence of unit vectors in $\mathbb{R}^2$ with $n$ even, then \[\Pr[\|\sigma_V\| \le r] \ge \frac{c}{n}.\]
	\end{prop}
	\begin{proof}
		After reordering the vectors of $V$ and possibly replacing them with their negations, we may assume $v_1=(1,0)$ and that $0\le \arg(v_1)\le \cdots \le \arg(v_n)\le \pi$.  Define $\tilde{V}=(\tilde{v}_1,\ldots,\tilde{v}_n)=(v_2,v_3,\ldots,v_{n-1},v_n,-v_1)$ and note that $\sigma_V$ and $\sigma_{\tilde{V}}$ have the same distribution. 
		
		Set $v_{n+1}=(-1,0)$. By applying Lemma~\ref{lem:circleOrder} to $V\cup \{v_{n+1}\}$, we see that $\sum_{i=1}^n \|v_i-v_{i+1}\|^2\le 4$, and hence by the pigeonhole principle, either
		\[\sum_{i=1}^{n/2} \|v_{2i-1}-v_{2i}\|^2\le 2,\ \ \ \mathrm{or}\ \ \ \sum_{i=1}^{n/2} \|v_{2i}-v_{2i+1}\|^2=\sum_{i=1}^{n/2}\|\tilde{v}_{2i-1}-\tilde{v}_{2i}\|^2\le 2.\]  Without loss of generality we will assume $\sum_{i=1}^{n/2} \|v_{2i-1}-v_{2i}\|^2\le 2$.  In this case, the result follows from Proposition~\ref{prop:pairing} by taking $\alpha$ to be sufficiently small in terms of
		\[r>\sqrt{2}\ge \sqrt{\sum_{i=1}^{n/2} \|v_{2i-1}-v_{2i}\|^2},\] by taking $\alpha=2\sqrt{2}(r-\sqrt{2})$, for example.

	\end{proof}
	
	In order to prove the optimal bound in Theorem~\ref{thm:main}, we will need some `stability analysis'; we will break our analysis into two cases depending on whether $V$ is `close to' the extremal example (of having $n/2$ copies of two vectors $(1,0)$ and $(0,1)$) or not. To this end, we make the following definitions.
	\begin{definition}
		We say that two unit vectors $v,x\in \RR^2$ are \textit{$\gamma$-close} for some real number $\gamma$ if either the angle between $v$ and $x$ is at most $\gamma$ radians or the angle between $-v$ and $x$ is at most $\gamma$ radians, and we say that the pair is \textit{$\gamma$-far} otherwise.  We say that a sequence of unit vectors $V=(v_1,\ldots,v_n)$ is \textit{$(2,\gamma)$-close} if there exist two unit vectors $x_1,x_2$ such that every $v_i$ is $\gamma$-close to either $x_1$ or $x_2$, and we say that $V$ is \textit{$(2,\gamma)$-far} otherwise.
	\end{definition}
	
	The following lemma supplies the main consequence of being $(2,\gamma)$-far that we require.
	
	\begin{lemma}\label{lem:farCharacterize}
		For any $\gamma>0$, if a sequence of vectors $V$ is $(2,\gamma)$-far, then there exist three vectors $u_1,u_2,u_3$ from $V$ such that every pair $u_i,u_j$ with $i\ne j$ is $\gamma$-far.
	\end{lemma}
	\begin{proof}
		Let $u_1$ be an arbitrary vector of $V$.  There must exist some vector $u_2$ of $V$ which is $\gamma$-far from $u_1$, as otherwise $x_1=x_2=u_1$ would contradict $V$ being $(2,\gamma)$-far.  Similarly, there must exist some $u_3$ which is $\gamma$-far from both $u_1,u_2$, as  otherwise $x_1=u_1$ and $x_2=u_2$ would contradict $V$ being $(2,\gamma)$-far.
	\end{proof}
	
	If $V$ is $(2,\gamma)$-far, then we will use the three vectors guaranteed by Lemma~\ref{lem:farCharacterize} and Lemma~\ref{lem:generalizedCircleOrder} to conclude that we can find a pairing of these vectors with pairwise distances strictly smaller than two.  To this end, we have the following lemma.
	\begin{lemma}\label{lem:threePointSet}
		Let $V'=(v_1',v_2',v_3',v_4')$ be a sequence of unit vectors in $\RR^2$ such that
		\begin{enumerate}
			\item $0=\arg(v_1')\le \arg(v_2')\le \arg(v_3')\le \arg(v_4')=\pi$, and
			\item $V'$ is $(2,\gamma)$-far for some $0<\gamma\le \pi/2$.
		\end{enumerate}
		Then
		\[\sum_{i=1}^3 \|v_i'-v'_{i+1}\|^2\le 4-8\sin^3(\gamma/2).\]
	\end{lemma}
	Here, and throughout this section, we use the fact that if $u,v$ are unit vectors at angle $\theta$ from each other, then
	\[\|u-v\|=2\sin(\theta/2).\]
	\begin{proof}
		By Lemma~\ref{lem:farCharacterize} we know that every pair of vectors besides $v_1',v_4'$ are $\gamma$-far.  As noted above, if $\theta_i$ is the angle between $v_i'$ and $v'_{i+1}$, then $\|v_i'-v'_{i+1}\|^2=2\sin^2(\theta_i/2)$.  Moreover, because $\theta_1+\theta_2+\theta_3=\pi$, a standard trigonometric identity implies that
		\[\sum_{i=1}^3 \|v_i'-v'_{i+1}\|^2=\sum_{i=1}^3 4\sin^2(\theta_i/2)=4(1-2\sin(\theta_1/2)\sin(\theta_2/2)\sin(\theta_3/2)).\]
		Because each $v_i',v'_{i+1}$ is $\gamma$-far, we have that $\gamma\le \theta_i\le \pi-\gamma$, giving the desired result.
	\end{proof}
	
	The lemmas above are enough to prove Theorem~\ref{thm:main} whenever $V$ is $(2,\gamma)$-far, so it remains to deal with the case that $V$ is $(2,\gamma)$-close.  This is easy to do in the following special case; here, we emphasize that our exact choices of $1/2$ and $\arcsin(0.1)$ are not particularly important.
	
	\begin{lemma}\label{lem:evenSteven}
		Let $V=(v_1,\ldots,v_n)$ be a sequence of unit vectors in $\RR^2$ with $n$ even and $0<\gamma\le \arcsin(0.1)$.  If there exists unit vectors $x_1,x_2$ and an even integer $m$ such that every $v_i$ with $i\le m$ is $\gamma$-close to $x_1$ and every $v_i$ with $i>m$ is $\gamma$-close to $x_2$, then
		\[\Pr[\|\sigma_V\|\le 1/2]=\Omega(n^{-1}).\]
	\end{lemma}
	\begin{proof}
		Possibly by taking negations of vectors, we may assume that every $v_i$ with $i\le m$ has angle at most $\gamma$ with $x_1$, and possibly by rotating and reordering the first $m$ vectors, we may further assume that $0=\arg(v_1)\le \cdots \le \arg(v_m)\le 2\gamma$. Applying a trivial bound alongside Lemma~\ref{lem:generalizedCircleOrder} with $V'=(v_1,v_m)$ gives
		\[\sum_{i=1}^{m/2}\|v_{2i-1}-v_{2i}\|^2\le \sum_{i=1}^{m-1} \|v_i-v_{i+1}\|^2\le \|v_1-v_m\|^2=4\sin^2(\arg(v_m)/2)\le 4\sin^2(\gamma)\le 0.1.\]
		By the same argument, we may negate and reorder the $v_i$ with $i>m$ so that
		\[\sum_{i=m/2+1}^{n/2}\|v_{2i-1}-v_{2i}\|^2\le 0.1,\]
		so in total we have
		\[\sum_{i=1}^{n/2}\|v_{2i-1}-v_{2i}\|^2\le 0.2.\]
		Since $r=1/2$ satisfies $r^2=1/4>0.2$, Proposition~\ref{prop:pairing} implies the result by taking $\alpha=0.01$, for example.
	\end{proof}
	Lemma~\ref{lem:evenSteven} solves (in a strong sense) the problem when $V$ is $(2,\gamma)$-close and an even number of vectors are close to each of $x_1,x_2$.   If instead an odd number of vectors are close to each of $x_1,x_2$, then we will prove the result by selecting two vectors $u_1,u_2$ near $x_1,x_2$ respectively, applying Lemma~\ref{lem:evenSteven} on $V-\{u_1,u_2\}$, and then adding signed copies of $u_1,u_2$ back to $\sigma_{V-\{u_1,u_2\}}$. The following geometric lemma will be necessary in order for this scheme to work.

	\begin{lemma}\label{lem:kiteBound}
		Given unit vectors $u,u'\in \RR^2$ and any vector $w\in \RR^2$ with $\|w\|\le 1/2$, there is a choice of signs $\eps,\eps'\in \{-1,+1\}$ such that $\|w+\eps u+\eps' u' \| \le \sqrt{2}$.
	\end{lemma}
	\begin{proof}
		Possibly by replacing $u'$ with its negation, we may assume that the angle $\beta\le \pi$ between $u$ and $u'$ is at least $\pi/2$, and possibly by rotating our vectors, we may assume without loss of generality that $u=(\cos(\beta/2),\sin(\beta/2))$ and $u'=(\cos(\beta/2),-\sin(\beta/2))$.
		
		Let $K=\|w\|\le 1/2$ so that $w=(K\cos(\theta),K\sin(\theta))$ for some $\theta$. Possibly by replacing $w$ with its negation, we may assume that $-\pi/2\le \theta\le \pi/2$, and without loss of generality, we may assume that $0\le \theta\le \pi/2$.
		Consider the vectors
		\[w_1=w-u-u'=(K\cos(\theta)-2\cos(\beta/2),K \sin(\theta)),\]
		\[w_2=w-u+u'=(K\cos(\theta),K\sin(\theta)-2\sin(\beta/2)),\]
		and observe that to prove the lemma, it suffices to show that at least one of these vectors has norm at most $\sqrt{2}$.  For this, we observe that
		\begin{align*}
			\|w_1\|^2+\|w_2\|^2 & =[K^2+4\cos^2(\beta/2)-4K\cos(\theta)\cos(\beta/2)]                   \\
			& \,\,\,\,\,\,\,\,\,+[K^2+4\sin^2(\beta/2)-4K\sin(\theta)\sin(\beta/2)] \\ &=2K^2+4-4K\cos(\theta-\beta/2)\le 4,\end{align*}
		where the last inequality relies on the fact that $K\le 1/2 \le 1/\sqrt{2}$, $|\theta-\beta/2|\le \pi/4$ for $0\le \theta\le \pi/2$, and $\pi/4\le \beta/2\le \pi/2$.  This implies that $\|w_t\|\le \sqrt{2}$ for some $t\in \{1,2\}$, proving the result.
	\end{proof}

	We can now prove Theorem~\ref{thm:main} in the case when $n$ is even.
	\begin{thm}\label{thm:even}
		There exists an absolute constant $c>0$ such that if $V=(v_1,\ldots,v_n)$ is a sequence of unit vectors in $\mathbb{R}^2$ with $n$ even, then
		\[\Pr\left[\|\sigma_V\|\le \sqrt{2}\right]\ge \frac{c}{n}.\]
	\end{thm}
	\begin{proof}
		The first half of this proof will parallel that of Proposition~\ref{prop:weak2}, and as such we omit some of the redundant details in this case.  Let $V=(v_1,\ldots,v_n)$ be a sequence of unit vectors in $\RR^2$ with $n$ even, and for concreteness, let $\gamma=\arcsin(0.1)\le \pi/2$ (though this exact value is not very important). We break our argument into two cases depending on whether $V$ is $(2,\gamma)$-close or not.
		
		First, we suppose that $V$ is $(2,\gamma)$-far. In this case, the following claim implies that we can apply Proposition~\ref{prop:pairing} with $r=\sqrt{2}$ and $\alpha=0.00001$ to get the desired result.
		\begin{claim}
			It is possible to reorder and negate some of the vectors of $V$ so that
			\[\sum_{i=1}^{n/2} \|v_{2i-1}-v_{2i}\|^2\le 1.9995.\]
		\end{claim}
		\begin{proof}
			By Lemma~\ref{lem:farCharacterize}, there exist $u_1,u_2,u_3$ in $V$ which are each $\gamma$-far from each other, and possibly by reordering, negating, and rotating these vectors, we can assume $v_1=u_1=(1,0)$ and that $0\le \arg(v_1)\le \cdots \le \arg(v_n)\le \pi$.  Letting $v_{n+1}=-v_1=(-1,0)$ and applying Lemma~\ref{lem:generalizedCircleOrder} with $V'=(v_1,u_2,u_3,v_{n+1})$ shows that
			\[\sum_{i=1}^{n} \|v_i-v_{i+1}\|^2\le 4-8\sin^3(\gamma/2)\le 2\cdot 1.9995.\]
			The claim follows from the pigeonhole principle by either considering $V$ or $\tilde{V}=(v_2,v_3,\ldots,v_n,-v_1)$.
		\end{proof}
		
		Next, suppose that $V$ is $(2,\gamma)$-close.  This in particular means that there exists unit vectors $x_1,x_2$ and some $1\le m\le n$ such that, possibly after reordering the vectors $V$, we have that $v_i$ is $\gamma$-close to $x_1$ for all $i\le m$ and $v_i$ is $\gamma$-close to $x_2$ for all $i>m$.  If $m$ is even, then the result follows from Lemma~\ref{lem:evenSteven}, so we can assume that $m$ is odd.  Let $V'$ be the subsequence of $V$ obtained by removing $v_m$ and $v_n$.  In this case $V'$ satisfies the conditions of Lemma~\ref{lem:evenSteven} with $x_1,x_2$ and $m-1$, so we conclude that $\Pr[\|\sigma_{V'}\|\le 1/2]=\Omega(n^{-1})$.  Observe that conditional on $\sigma_{V'}$ lying in this range, the probability that $\sigma_V=\sigma_{V'}+\eps_{m}v_m+\eps_nv_n$ has norm at most $\sqrt{2}$ is at least $1/4$ by Lemma~\ref{lem:kiteBound}, so we again conclude that $\Pr[\|\sigma_V\|\le \sqrt{2}]=\Omega(n^{-1})$, completing the proof.
	\end{proof}
	
	We will now deduce the case of odd $n$ from the even case, for which we need the following geometric result.
	\begin{prop}\label{prop:reduceToEven}
		If $V=(v_1,\ldots,v_n)$ is a sequence of unit vectors in $\mathbb{R}^2$ with $n\ge 3$, then (at least) one of the following statements holds.
		\begin{enumerate}
			\item[(a)] There exists some $i$ such that for all $j$, either \[\arg(v_i)\le \arg(v_j)\le \arg(v_i)+{7\pi}/{24},\] or \[\arg(v_i)\le \arg(-v_j)\le \arg(v_i)+{7\pi}/{24}.\]
			\item[(b)] There exist distinct $i,j,k$ such that for any $w\in \mathbb{R}^2$ with $\|w\|\le \sqrt{2}$, there exist signs $\epsilon_{i},\epsilon_{j},\epsilon_{k}\in \{-1,+1\}$ such that $\|w+\epsilon_{i} v_{i}+\epsilon_{j} v_{j}+\epsilon_{k}v_{k}\|\le \sqrt{2}$.
		\end{enumerate}
	\end{prop}
	We note again that the exact value of ${7\pi}/{24}$ is not crucial here; we simply need some number strictly smaller than ${\pi}/{3}$ and slightly larger than ${\pi}/{4}$.
	\begin{proof}
		Our proof rests on the following technical geometric claim analogous to Lemma~\ref{lem:kiteBound}.
		\begin{claim}\label{cl:geometry}
			If $u,u'\in \mathbb{R}^2$ are unit vectors at an angle $\beta$ satisfying ${\pi}/{2}\le \beta\le {17\pi}/{24}$, then for any $w'\in \mathbb{R}^2$ of norm at most $\sqrt{3}$,  there exist $\epsilon,\epsilon'\in \{-1,+1\}$ such that $\|w'+\epsilon u+\epsilon' u'\|\le \sqrt{2}$.
		\end{claim}
		\begin{proof}
			Possibly by rotating our vectors, we may assume without loss of generality that $u=(\cos(\beta/2),\sin(\beta/2))$ and $u'=(\cos(\beta/2),-\sin(\beta/2))$.  Let $K=\|w'\|\le \sqrt{3}$ so that $w'=(K\cos(\theta),K\sin(\theta))$ for some $\theta$.  Possibly by replacing $w'$ with its negation we may assume $-{\pi}/{2}\le \theta\le {\pi}/{2}$, and without loss of generality, we may assume that $0\le \theta\le {\pi}/{2}$.  Consider the pair of vectors
			\[w_1=w'-u-u'=(K\cos(\theta)-2\cos(\beta/2),K \sin(\theta)),\]
			\[w_2=w'-u+u'=(K\cos(\theta),K\sin(\theta)-2\sin(\beta/2)),\]
			and observe that for the claim it suffices to show that at least one of these vectors has norm at most $\sqrt{2}$.
			
			First, consider the case that $K\le 2 \cos(\theta-\beta/2)$. As in the argument in Lemma~\ref{lem:kiteBound}, we have by our assumption on $K$ that
			\begin{align*}\|w_1\|^2+\|w_2\|^2&=2K^2+4-4K\cos(\theta-\beta/2)\le 4,\end{align*}
			and hence $\| w_t\|^2\le 2$ for some $t\in \{1,2\}$ as desired.
			
			Now, assume that $K>2 \cos(\theta-\beta/2)$; in this case, we shall show $\|w_1\|\le \sqrt{2}$.  Because $K\le \sqrt{3}$, our assumed inequality implies $|\theta-\beta/2|>{\pi}/{3}$.  Note that we can not have $\theta> \beta/2+{\pi}/{3}$ since $\beta/2\ge {\pi}/{4}$ and $\theta\le {\pi}/{2}$, so we must have
			\[0\le \theta<\beta/2-\frac{\pi}{3}\le \frac{\pi}{48},\]
			with this last step using $\beta\le {17\pi}/{24}$.
			For any such $\theta$ and ${\pi}/{2}\le \beta\le  {17\pi}/{24}$ and $K\le \sqrt{3}\le 1.76$, we have
			\begin{align*}\|w_1\|^2&=K^2+4\cos^2(\beta/2)-4K\cos(\theta)\cos(\beta/2)\\ &\le K^2+4\cos^2\left(\frac{\pi}{4}\right)-4K\cos\left(\frac{\pi}{48}\right)\cos\left(\frac{17\pi}{48}\right)\le K^2+2-1.76 K\le 2,\end{align*}
			proving the claim.
		\end{proof}
		
		We also need the following observation.
		\begin{claim}\label{cl:largeAngle}
			If (a) does not hold, then there exists some $i,j$ such that the (shortest) angle between $v_i$ and both of $v_j,-v_j$ is at least ${7\pi}/{24}$.
		\end{claim}
		\begin{proof}
			If this were not the case, then we can assume, possibly after replacing some vectors with their negations, that every vector has angle at most ${7\pi}/{24}$ with $v_1$, and possibly by rotating all of our vectors, we can assume $\arg(v_1)={7\pi}/{24}$.  If we let $v_i$ be such that $\arg(v_i)=\min_k \arg(v_k)$ and $v_j$ be such that $\arg(v_j)=\max_k \arg(v_k)$, then we must have $\arg(v_j)\ge \arg(v_i)+ {7\pi}/{24}$ (since otherwise (a) would hold for $i$ by the definition of $i,j$).  We also have
			\[\arg(v_j)\le {7\pi}/{24}+\arg(v_1)\le {14\pi}/{24}+\arg(v_i)\le \pi+\arg(v_i)\]
			by the assumption on $v_1$, which implies that the angle between $v_i$ and $v_j$ is $\arg(v_j)-\arg(v_i)\ge {7\pi}/{24}$.  Similarly, because $\arg(v_j)\le {7\pi}/{24}+\arg(v_1)\le \pi$ we find that $\arg(-v_j)=\pi+\arg(v_j)$.  This implies that $\pi\le \arg(-v_j)-\arg(v_i)\le \pi+{14\pi}/{24}$, which implies that the angle between $-v_j$ and $v_i$ is $2\pi -\arg(-v_j)+\arg(v_i)\ge \pi-{14\pi}/{24}\ge {7\pi}/{24}$ as desired.
		\end{proof}
		
		We now complete the proof.  Assume that (a) does not hold and let $i,j$ be as in Claim~\ref{cl:largeAngle}, and let $k$ be any index not equal to $i,j$.  We claim that (b) holds with this choice of $i,j,k$. Let $w\in \mathbb{R}^2$ be an arbitrary vector of norm at most $\sqrt{2}$. Observe that there exists some $\epsilon_k \in \{-1,1\}$ such that $w'=w+\epsilon_k v_k$ has norm at most $\sqrt{3}$;  indeed, this follows by taking any $\epsilon_k$ such that $w$ and $\epsilon_k v_k$ have angle at most ${\pi}/{2}$ between them.  Possibly by replacing $v_j$ with its negation, we may assume that the (shortest) angle $\beta$ between $v_i$ and $v_j$ satisfies $\beta\ge {\pi}/{2}$, and by our choice of $i,j$, we must have $\beta\le {17\pi}/{24}$ (as otherwise, the angle between $v_i$ and $-v_j$ would be at most ${7\pi}/{24}$).  Applying Claim~\ref{cl:geometry} with $u=v_i,u'=v_j$ gives signs with the desired property, finishing the proof.
	\end{proof}
	We now have all that we require to prove Theorem~\ref{thm:main}.
	\begin{proof}[Proof of Theorem~\ref{thm:main}]
		Let $V=(v_1,\ldots,v_n)$ be a sequence of unit vectors in $\mathbb{R}^2$ with\footnote{We leave the  $n=1$ case as an exercise to the reader.} $n\ge 2$.  The result holds if $n$ is even by Theorem~\ref{thm:even}, so we may assume that $n\ge 3$ is odd.
		
		First, consider the case that Proposition~\ref{prop:reduceToEven}(a) applies to $V$, and possibly by rotating and reordering our vectors we can assume $0= \arg(v_1)\le \cdots \le \arg(v_n)\le {7\pi}/{24}$. By Lemma~\ref{lem:generalizedCircleOrder}, we have
		\[\sum_{i=1}^{n-1} \|v_i-v_{i+1}\|^2\le \|v_1-v_n\|^2\le 2\sin^2\left(\frac{7\pi}{48}\right)\le 1/2.\]
		By Proposition~\ref{prop:pairing}, we then have that $\|\sigma_{V-\{v_n\}}\|\le 1$ occurs with probability $\Omega(n^{-1})$, and conditional on this event, we have with probability at least $1/2$ that $\|\sigma_
		{V-\{v_n\}}+\eps_n v_n\|\le \sqrt{2}$ (since $\sigma_{V-\{v_n\}}$ and $\eps_n v_n$ will be at angle at least $\pi/2$ from each other with probability at least $1/2$), proving the result in this case.
		
		Next, assume that Proposition~\ref{prop:reduceToEven}(b) applies for some $i,j,k$, and let $V'=V-\{v_{i}-v_{j}-v_k\}$.  By Theorem~\ref{thm:even} we have $\|\sigma_{V'}\|\le \sqrt{2}$ with probability $\Omega(n^{-1})$, and conditional on this event, we have by Proposition~\ref{prop:reduceToEven}(b) that $\| \sigma_{V'}+\eps_{i}v_{i}+\eps_{j}v_{j}+\eps_k v_k\|\le \sqrt{2}$ with probability at least $1/8$, proving the result in this case and hence completing the proof.
	\end{proof}
	
	As an aside, we note that our approach here can be used to obtain results for the reverse Littlewood--Offord problem in $\RR^d$ in general. In particular, similar geometric arguments can be used to show that for every $d\ge 2$, there exist absolute constants $r,c>0$ depending only on $d$ such that such that for any unit vectors $v_1,\ldots,v_n\in \mathbb{R}^d$ and independent Rademacher random variables $\eps_1,\dots,\eps_n$, we have $\Pr[||\epsilon_1 v_1+\ldots+\epsilon_n v_n||_2 \leq r]\geq n^{-d^2/4}.$. However, we do not know how to use these ideas to obtain the same tight results of (Beck's) Theorem~\ref{thm:Beck}.
	
	\section{Conclusion}\label{sec:conc}
	In this paper we gave a new proof of an old conjecture of Erd\H{o}s by showing that if $v_1,\ldots,v_n\in \mathbb{R}^2$ are unit vectors, then with probability $\Omega(n^{-1})$, their random signed sum has norm at most $\sqrt{2}$.  The radius $\sqrt{2}$ is best possible here for even $n$ by considering the case $v_i=(1,0)$ for an odd number of $i$ and $v_i=(0,1)$ otherwise, but as far as we know the following stronger bound might hold for odd $n$, matching the original conjecture of Erd\H{o}s.
	\begin{conj}
		There exists an absolute constant $c>0$ such that for any unit vectors $v_1,\ldots,v_n\in \mathbb{R}^2$ with $n$ odd and independent Rademacher random variables $\eps_1,\dots,\eps_n$, we have
		$$\Pr\left[\|\epsilon_1 v_1+\ldots+\epsilon_n v_n\|_2 \leq 1\right]\geq \frac{c}{n}.$$
	\end{conj}
	Our present proof of Theorem~\ref{thm:main} admits a bit of slack when $n$ is odd and can be adjusted to prove $\Pr[\|\sigma_V\|\le r]=\Omega(n^{-1})$ for some $r<\sqrt{2}$, though it seems new ideas are needed to get all the way down to $r=1$. The order of magnitude of $\Omega(n^{-1})$ in Theorem~\ref{thm:main} is best possible, but exactly determining the implicit constant seems to be an intriguing problem in discrete geometry.
	\begin{quest}
		For $n\ge 1$, let $V$ range over all sequences of $n$ unit vectors in $\mathbb{R}^2$. If $r > 0$, how does the function
		\[f(r)=\liminf_{n\rightarrow \infty} \inf_V \Pr[\|\sigma_V\| \le r]n\]
		behave?  In particular, is $f(r)$ always an integer multiple of ${4}/{\pi}$?
	\end{quest}
	Theorem~\ref{thm:main} shows that $f(r) > 0$ if and only if $r\ge \sqrt{2}$. Note that the `in particular' part of this question would hold if the minimizer of the probability always consisted of roughly $n/2$ copies of $(1,0)$ and $(0,1)$.  For the specific case of $r=\sqrt{2}$ we believe the following even stronger statement holds.
	\begin{conj}
		For all $n$ sufficiently large, there exists some $t\le n$ such that for any unit vectors $v_1,\ldots,v_n\in \mathbb{R}^2$ and independent Rademacher random variables $\eps_1,\dots,\eps_n$, we have
		$$\Pr\left[||\epsilon_1 v_1+\ldots+\epsilon_n v_n||_2 \leq \sqrt{2}\right]\geq \Pr\left[||\epsilon_1 v_1'+\ldots+\epsilon_n v_n'||_2 \leq \sqrt{2}\right],$$
		where $v'_i=(1,0)$ for $i\le t$ and $v'_i=(0,1)$ for $i>t$.
	\end{conj}
	
	\section*{Acknowledgements}
	We thank Noga Alon, Poornima Belvotagi, Timothy Chu, Jacob Fox, Zachary Hunter, Huy Pham and Shengtong Zhang for helpful conversations. The first author was supported by NSF grant DMS-2103154, the third author was supported by NSF grant DMS-2237138 and a Sloan Research Fellowship, and the fourth author was supported by the NSF postdoctoral research fellowship under grant DMS-2202730.

	\bibliographystyle{amsplain}
	\bibliography{refs}
	
\end{document}